\documentclass[10pt]{article}

\usepackage[latin1]{inputenc}
\usepackage[english]{babel}
\usepackage[T1]{fontenc}
\usepackage{lmodern}
\usepackage{graphicx, xcolor}
\usepackage{amsfonts, ulem}
\usepackage{stmaryrd}
\usepackage[nice]{nicefrac}
\usepackage{amsmath,amsthm,amssymb,mathabx}
\usepackage{enumerate}
\usepackage{hyperref}
\usepackage[top=2cm, bottom=2cm, left=2cm , right=2cm]{geometry}
\usepackage{caption}
\usepackage{xcolor}
\usepackage{dsfont}

\usepackage{caption}

\newtheorem{thm}{Theorem}
\newtheorem{df}{Definition}
\newtheorem{lem}{Lemma}
\newtheorem{prop}{Proposition}
\newtheorem{cor}{Corollary}

\newenvironment{rmq}{\noindent{\bf Remark:} }{}

\newcommand{\Cov}{\mathtt{Cov}}

\title{\bf A Generalized Urn Model with \\Multiple Drawing \\and Random
Addition}
\author{Aguech Rafik.\\
\small Departement of Mathematics\\[-0.8ex]
\small King Saoudian University\\[-0.8ex]
\small Erryadh K.S.A.\\
\small\tt rafik.aguech@ipeit.rnu.tn\\
\and
Lasmar Nabil \\
\small Departement des Mathematiques \\[-0.8ex]
\small Institut Pr\'eparatoire aux \\[-0.8ex]
\small Etudes d'ing\'enieurs \\[-0.8ex]
\small Monastir, Tunisia\\
\small\tt nabillasmar@yahoo.fr\\
\and
Selmi Olfa\\
\small D\'epartement des Mathematiques\\[-0.8ex]
\small Facult\'e des Sciences de Monastir\\[-0.8ex]
\small Monastir, Tunisia\\
\small\tt selmiolfa3@yahoo.fr}

\begin{document}
\title{Unbalanced urn model with random addition }
\author{Rafik Aguech \thanks{D\'epartement des Math\'ematiques, King saoudian university,}
Nabil Lasmar\thanks{D\'epartement des Math\'ematiques, Institut
Pr\'eparatoire aux \'Etudes d'IngÃ©nieur, Monastir, Tunisie
(\texttt{nabillasmar@yahoo.fr}).}, and Olfa
Selmi\thanks{D\'epartement des Math\'ematiques, Universit\'e des
Sciences de Monastir, Monastir, Tunisia
(\texttt{selmiolfa3@yahoo.fr}).}} \maketitle

\begin{abstract}
 In this paper, we consider a multi-drawing urn model
with random addition. At each discrete time step, we draw a sample
of $m$ balls. According to the composition of the drawn colors, we
return the balls together with a random number of balls depending
on two discrete random variables $X$ and $Y$ with finite means and
variances. Via the stochastic approximation algorithm, we give
limit theorems describing the asymptotic behavior of white
balls.
\end{abstract}

\section{Introduction}


The classical P\'olya urn was introduced by P\'olya and
Eggenberger \cite{Polya} describing  contagious diseases.
 The first model  is as
follows: An urn contains  balls of two colors at the start, white
and black. At each step, one picks  a ball randomly and returns it
to the urn with a ball of the same color.

Afterward this model was generalized  and it has become a simple
tool to describe several  models such finance, clinical trials
(see \cite{Pages}, \cite{Wei}), biology (see \cite{bio}), computer
sciences, internet (see
\cite{Mahmoud},\cite{Goldman}), etc.  \\
Recently, H. Mahmoud, M.R. Chen, C.Z Wei, M. kuba and H. Sulzbach
\cite{Kuba-Mahmoud-Panholzer,Chen-Kuba,Chen-Wei,kuba-Zulzbach,Kuba-mahmoud,Kuba-Mahmoud2},
have focused on the multidrawing urn. Instead of picking a ball,
one picks a sample of $m$ balls ($m\geq 1$), say $l$ white and
$m-l$ black balls. the pick is returned back to the urn together
with $a_{m-l}$ white and $b_{l}$ black balls, where $a_l$ and
$b_l, 0\leq l\leq m$ are integers. At first, they treated two
particular cases when
  \{$a_{m-l}=c\times l \quad \text{and}\quad
b_{m-l}=c\times (m-l)$\} and when \{$a_{m-l}=c\times (m-l)$
\quad\text{and}\quad $b_{m-l}=c\times l$\}, where $c$ is a
positive constant. By different methods as martingales and moment
methods, the authors described the asymptotic behavior of the urn
composition. When considering the general case and in order to
ensure the existence of a martingale, they
 supposed that $W_n$, the number of white balls in
the urn after $n$ draws, satisfies the affinity condition i.e,
there exists two deterministic sequences $(\alpha_n)$ and
$(\beta_n)$ such that, for all $n\geq 0$,
$\mathbb{E}[W_{n+1}|\mathcal{F}_n]=\alpha_n W_{n}+\beta_n$. Under
this condition, the authors focused on small and large index urns.
Later, the affinity condition was removed in the work of C.
Mailler, N. Lasmer and S. Olfa  \cite{C.N.O}, they generalized
this model and looked at the case of more than two
colors.\\

In the present paper, we deal with an unbalanced urn model, which
was not been sufficiently addressed in the literature. It was
mainly dealt with in the works of R. Aguech \cite{R.Aguech},  S.
Janson \cite{S. Janson} and H. Renlund \cite{Renlund1, Renlund2}.
In \cite{R.Aguech} and \cite{S. Janson}, the authors dealt with
model with a simple pick, whereas in \cite{Renlund1,Renlund2} the
author considered a model with two picks and, under some
conditions, they described the asymptotic behavior of the urn
composition.

In this paper, we aim to give a generalization of a recent work
\cite{A.L.O}. We deal with an unbalanced urn model with random
addition. We consider an urn containing two different colors white
and blue. We suppose that the urn is non empty at time 0. Let
denote by $W_n$ (resp $B_n$) the number of white balls (resp blue
balls) and by $T_n$ the total number of balls in the urn at time
$n$. Let $(X_n)_{n\geq 0}$ and $(Y_n)_{n\geq 0}$ be strictly
positive sequences of independent identically distributed discrete
random variables with finite means and variances. The model we
study is defined as follows: At a discrete time, we pick out a
sample of $m$ balls from the urn (we suppose that $T_0=W_0+B_0\geq
m$) and according to the composition of the sample, we return the
balls with $Q_n(\xi_n,m-\xi_n)^t$ balls, where $Q_n$ is a $2\times
2$ matrix depending on the variables $X_n$ and $Y_n$ and  $\xi_n$
is the number of white balls in the $n^{th}$ sample.

 Let $(\mathcal{F}_n)_{n\ge 0}$ be the $\sigma$-field
 generated by the first $n$ draws. We summarize
the evolution of
       the urn by
       the recurrence
\begin{equation}\label{recurrence}
  \begin{pmatrix}
    W_{n} \\
    B_{n} \\
  \end{pmatrix}\stackrel{\mathcal D}{=}
  \begin{pmatrix}
    W_{n-1} \\
    B_{n-1} \\
  \end{pmatrix}+Q_n
  \begin{pmatrix}
    \xi_{n} \\
  m-\xi_{n} \\
  \end{pmatrix}.
\end{equation}
 Note that, with these notations, we have
 \begin{equation*}\mathbb{P}[\xi_n=k|\mathcal{F}_{n-1}]=\displaystyle\frac{\binom {W_{n-1}} k \binom {B_{n-1}} {m-k}}{\binom {T_{n-1}} m}.\end{equation*}

 The paper is organized as follows. In Section \ref{Main results},
 we give the main results of the paper. In the first paragraph of Section \ref{Proofs}, we develop
 Theorem 1 \cite{Renlund1} and apply it to our urn model. The rest
 of this section is devoted to the prove the theorems.

 \textbf{Notation:} For a random variable $R$, we denote
 by $\mu_R=\mathbb{E}(R)$ and
 $\sigma_R^2=\mathbb{V}ar(X)$. Note that $\mu_X, \mu_Y ,\sigma^2_X$
 and $\sigma^2_Y$ are finite.\\

\section{Main Results}\label{Main results}

\begin{thm}\label{thmXopp}
Consider the urn model evolving by the matrix $Q_n=
\begin{pmatrix}
  0 & X_n \\
  X_n & 0 \\
\end{pmatrix}$. We have the following results:

\begin{enumerate}
    \item \begin{equation}\label{asymp-T_n}T_n\stackrel{a.s}{=}m\mu_X n
+o(\sqrt{n}\ \ln(n)^\delta),\end{equation} \begin{equation}
W_n\stackrel{a.s}{=}\frac{m\mu_X }{2}n+o(\sqrt{n}\
\ln(n)^{\delta})\quad \text{and}\quad
B_n\stackrel{a.s}{=}\frac{m\mu_X }{2}n+o(\sqrt{n}\
\ln(n)^{\delta});\quad\delta>\frac{1}{2}.\end{equation}

    \item \begin{equation}\frac{W_n-\frac{1}{2}T_n}{\sqrt{n}}
\stackrel{\mathcal{L}}{\longrightarrow}\mathcal{N}\Big(0,\frac{m(\sigma_X^2+\mu_X^2)}{12}\Big).\end{equation}
    \item \begin{equation}\frac{W_n-\mathbb{E}(W_n)}{\sqrt{n}}\stackrel{\mathcal{L}}{\longrightarrow}
\mathcal{N}\Big(0,\frac{m(\sigma_X^2+\mu_X^2)+m^2\sigma^2_X}{12}\Big).\end{equation}
\end{enumerate}

\end{thm}

\begin{thm}\label{thmXself}
Consider the urn model evolving by the matrix $Q_n=
\begin{pmatrix}
  X_n & 0 \\
  0 & X_n \\
\end{pmatrix}$. There exists a positive random variable $\tilde W_\infty$, such that
\begin{equation}T_n\stackrel{a.s}{=}m\mu_X n
+o(\sqrt{n}\ \ln(n)^\delta),\quad W_n\stackrel{a.s}{=}\tilde
W_\infty n +o(n) \quad\mbox{and}\;\;B_n\stackrel{a.s}{=}
(m\mu_X-\tilde W_\infty)n+o(n).
\end{equation}
\end{thm}
\begin{rmq}
The random variable $\tilde W_\infty$ is absolutely continuous
whenever $X$ is bounded.
\end{rmq}

\begin{thm}\label{thmXYopp} Consider the urn model evolving by
the matrix $Q_n=
\begin{pmatrix}
  0 & X_n \\
  Y_n & 0 \\
\end{pmatrix}.$ Let $z:=\frac{\sqrt{\mu_X}}{{\sqrt{\mu_X}+\sqrt{\mu_Y}}}$, we
have the following results:

\begin{enumerate}

    \item \begin{equation}\label{SL-Total}T_n\stackrel{a.s}{=}m\sqrt{\mu_X}\sqrt{\mu_Y}\
    n+o(n),\end{equation}
    \begin{equation} W_n\stackrel{a.s}{=}m\sqrt{\mu_X}\sqrt{\mu_Y}\ z\  n+o(n)\quad\text{and}\quad B_n\stackrel{a.s}{=}m\sqrt{\mu_X}\sqrt{\mu_Y}(1-z)\ n+o(n).\end{equation}

    \item \begin{equation}\frac{W_n-z T_n}{\sqrt{n}}\stackrel{\mathcal{L}}{\longrightarrow}\mathcal{N}\Big(0,\frac{ G(z)}{3}\Big),\end{equation}

where, \begin{equation*} G(x)=\sum_{i=0}^4a_ix^i,\end{equation*}
with
\begin{eqnarray*}a_0=m^2(\sigma^2_X+\mu_X^2)&,&a_1=m(1-2m)(\sigma_X^2+\mu_X^2),\\
a_2=3m(m-1)(\sigma_X^2+\mu_X^2)-2m(m-1)\mu_X\mu_Y &,&
a_3=m\mathbb{E}(X-Y)^2-2(m^2-m)\bigl(\sigma_X^2+\mu_X^2-\mu_X\mu_Y\bigr)\\
\text{and}\quad a_4=m(m-1)\mathbb{E}(X-Y)^2.\end{eqnarray*}

\end{enumerate}
\end{thm}

\begin{thm}\label{thmXYself}
Consider the urn evolving by the matrix $Q_n=
\begin{pmatrix}
  X_n & 0 \\
  0 & Y_n \\
\end{pmatrix}.$ We have the following results:

\begin{enumerate}
    \item If $\mu_X > \mu_Y$,
\begin{equation}
T_n\stackrel{a.s}{=}m\mu_Xn+o(n),\quad
W_n\stackrel{a.s}{=}m\mu_Xn+o(n)\quad \text{and}\quad
B_n\stackrel{a.s}{=}B_{\infty}n^\rho+o(n^\rho),
\end{equation}
where $\rho=\frac{\mu_Y}{\mu_X}$ and $B_{\infty}$ is a positive
random variable.

    \item If $\mu_X=\mu_Y$,
\begin{equation}T_n\stackrel{a.s}{=}m\mu_Xn+o(n),\quad W_n\stackrel{a.s}{=}W_{\infty}n+o(n)\quad \text{and}\quad B_n\stackrel{a.s}{=}(\mu_Xm-W_{\infty})\ n+o(n),\end{equation}
where $W_\infty$ is a positive random variable.
\end{enumerate}

\end{thm}
\begin{rmq}
The case when $\mu_X<\mu_Y$ is obtained by interchanging the
colors.
\end{rmq}\\
\textbf{Example:} Let $m=1$, this particular case was studied by
R. Aguech \cite{R.Aguech}.
 Using martingales and branching processes , R. Aguech proved the following results:\\
if $\mu_X>\mu_Y$,
\begin{equation*}W_n=\mu_X n+o(n),\quad B_n=D n^\rho\quad \text{and}\quad T_n=\mu_X n+o(n),\end{equation*}
where $D$ is a positive random variable.\\
If $\mu_X=\mu_Y$,
\begin{equation*}W_n=\mu_X \frac{W}{W+B}n+o(n)\quad \text{and}\quad B_n=\mu_X\frac{B}{W+B}n+o(n),\end{equation*}
where $W$ and $B$ are positive random variables obtained by
embedding some martingales in continuous time.

\section{Proofs}\label{Proofs}

The stochastic algorithm approximation plays a crucial role in the
proofs in order to describe the asymptotic composition of the urn.
As many versions of the stochastic algorithm exist in the
literature (see \cite{Duflo} for example), we adapt the version of
H. Renlund in \cite{Renlund1, Renlund2}.
\subsection{A basic tool: Stochastic approximation}

\begin{df}\label{def-algo}
A stochastic approximation algorithm $(U_n)_{n\geq 0}$ is a
stochastic process taking values in $[0,1]$ and adapted to a
filtration $\mathcal{F}_n$ that satisfies
\begin{equation}\label{eq:algo_sto}
U_{n+1}-U_n = \gamma_{n+1}\big(f(U_n)+\Delta M_{n+1}\big),
\end{equation}
where $(\gamma_n)_{n\geq 1}$ and $(\Delta_n)_{n\geq 1}$ are two
$\mathcal F_n$-measurable sequences of random variables, $f$ is a
function from $[0,1]$ onto $\mathbb R$ and the following
conditions hold almost surely.

\begin{description}
    \item [(i)]$\frac{c_l}{n}\leq \gamma_n \leq \frac{c_u}{n}$,
    \item [(ii)]$|\Delta M_n|\leq K_u,$
    \item [(iii)]$ |f(U_n)|\leq K_f,$
    \item [(iv)]$ \mathbb E[\gamma_{n+1}
\Delta M_{n+1}| \mathcal F_n] \leq K_e \gamma_n^2,$
\end{description}

where the constants $c_l, c_u, K_u, K_f, $ and $K_e$ are positive
real numbers.
\end{df}
\begin{df} Let
$Q_f=\{x; f(x)=0\}.$ A zero $p\in Q_f$ will be called stable if
there exists a neighborhood $\mathcal{N}_p$ of $p$ such that
$f(x)(x-p)<0$ whenever $x\in \mathcal{N}_p\setminus\{p\}.$ If $f$
is differentiable, then $f'(p)$ is sufficient to determine that
$p$ is stable.
\end{df}

\begin{thm}[\cite{Renlund1}]\label{th:renlund}Let $U_n$ bea
stochastic algorithm defined in Equation (\ref{eq:algo_sto}). If
$f$ is continuous, then $\displaystyle\lim_{n\rightarrow +\infty}
U_n$ exists almost surely and is in $Q_f$. Furthermore, if $p$ is
a stable zero, then $\mathbb{P}\Big(U_n\longrightarrow p\Big)>0.$
\end{thm}
\begin{rmq}
The conclusion of Theorem \ref{th:renlund} holds if we replace the
condition $(ii)$ in Definition \ref{def-algo} by the following
condition $\mathbb{E}[\Delta M_{n+1}^2|\mathcal{F}_n]\leq K_u$.

\end{rmq}

\begin{proof}[Proof of Theorem \ref{th:renlund}]
For the convenience of the reader, we adapt the proof of Theorem
\ref{th:renlund} and we show that, under the new condition $(ii)
 \quad\mathbb{E}[\Delta M_{n+1}^2|\mathcal{F}_n]\leq
 K_u$, the conclusion remains true. In fact, the following lemmas are useful.\\
\begin{lem}\label{V_n converges}
Let $V_n=\sum_{i=1}^n\gamma_i\Delta M_i$. Then, $V_n$ converges
almost surely.
\end{lem}
\begin{proof}
Set $A_i=\gamma_i\Delta M_i$ and $\tilde
A_i=\mathbb{E}[A_i|\mathcal{F}_{i-1}].$ Define the martingale
$C_n=\sum_{i=1}^n(A_i-\tilde A_i),$ then
\begin{eqnarray*}
\mathbb{E}(C_n^2)&\leq&\sum_{i=1}^n\mathbb{E}(A_i^2)=\sum_{i=1}^n\mathbb{E}(\gamma_i^2\Delta M_i^2)\\
&\leq&\sum_{i=1}^n\frac{c_u^2}{i^2}\mathbb{E}(\Delta M_i^2),
\end{eqnarray*}
if there exists some positive constant $K_u$ such that
$\mathbb{E}[\Delta M_{n+1}^2|\mathcal{F}_n]\leq K_u$, we
conclude that $C_n$ is an $L^2-$ martingale and thus converges almost surely.\\
Next, since \[\sum_{i\geq 1}|\tilde A_i|\leq\sum_{i\geq
1}\frac{c_u^2}{(i-1)^2}K_l<+\infty,\] the series $\sum_{i\geq
1}A_i$ must also converges almost surely.
\end{proof}
\begin{lem}\label{Q_f}Let $U_\infty$ be the set of accumulation point of
$\{U_n\}$ and  $Q_f=\{x; f(x)=0\}$ be the zeros of $f$. Suppose
$f$ is continuous. Then,
\[\mathbb{P}\Big(U_\infty \subseteq Q_f\Big)=1.\]

\end{lem}
\begin{proof}
See \cite{Renlund1}
\end{proof}

 Next, we prove the main result of the theorem.
 If $\displaystyle\lim_{n\rightarrow +\infty} U_n$ does not exist, we can find two rational numbers in
 the open interval $]\displaystyle\liminf_{n\rightarrow +\infty} U_n, \displaystyle\limsup_{n\rightarrow +\infty} U_n[$.\\
 Let $p<q$ be two arbitrary different rational numbers. If we can
 show that
 \[\mathbb{P}\Big(\{\liminf U_n \leq p\}\cap\{\limsup U_n \geq q\}\Big)=0,\]
 then, the existence of the limit will be established and the claim of
 the theorem follows from Lemma \ref{Q_f}.\\
 For this reason, we need to distinguish two different cases whether or not $p$
 and $q$ are in the same connected component of $Q_f$.\\

 \textbf{Case 1:
  $p$ and $q$ are not in the same connected
 component of $Q_f$.}\\
  See the proof in \cite{Renlund1}.\\

 \textbf{Case 2:
 $p$ and $q$ are in the connected component of
 $Q_f$.}\\
 Let $p$ and $q$ be two arbitrary rational numbers such that $p$ and
$q$ are in the same connected component of $Q_f$.
 Assume that $\displaystyle\liminf_{n\rightarrow +\infty} U_n \leq p$ and fix an
arbitrary $\varepsilon$ in such a way that $0\leq \varepsilon \leq
q-p$.\\
We aim to show that $\displaystyle\limsup_{n\rightarrow +\infty}
U_n\leq q$ i.e, it is sufficient to
show that $\displaystyle\limsup_{n\rightarrow +\infty} U_n\leq p+\varepsilon.$\\
In view of Lemma \ref{V_n converges}, we have
$V_n=\sum_{i=1}^n\gamma_i\Delta M_i$ converges a.s, then, for a
stochastic $N_1> 0$, for $n,m> N_1$ we have
$|W_n-W_m|<\frac{\varepsilon}{4}$ and
$\gamma_n\Delta M_n\leq \frac{\varepsilon}{4}$.\\
Let $N=max(\frac{4K_f}{\varepsilon},N_1)$. By assumption, there is
some stochastic $n>N$ such that $U_n-p<
\frac{\varepsilon}{2}$.\\
Let
\[\tau_1=\inf\{k\geq n; U_k\geq p\}\quad \text{and}\quad \sigma_1=\inf\{k>\tau_1; U_k<p\},\]
and define, for $n\geq 1,$
\[\tau_{n+1}=\inf \{k>\sigma_n; U_k\geq p\}\quad \sigma_{n+1}=\inf \{k>\tau_n; U_k<p\}.\]
Now, for all $k$ we have
\[U_{\tau_k}=U_{\tau_k-1}+\gamma_{\tau_k-1}(f(U_{\tau_k-1})+\Delta M_{\tau_k}).\]
Recall that $\gamma_{\tau_k-1}f(X_{\tau_k-1})\leq
\frac{K_f}{\tau_{k}-1}\leq \frac{K_f}{n}$, for $n\geq N\geq
\frac{4K_f}{\varepsilon}$ we have
$\gamma_{\tau_k-1}f(X_{\tau_k-1})<\frac{\varepsilon}{4}$. It
follows,
\[\gamma_{\tau_k-1}(f(U_{\tau_k-1})+\Delta M_{\tau_k})\leq \frac{K_f}{n}+\frac{\varepsilon}{4}\leq
\frac{\varepsilon}{4}+\frac{\varepsilon}{4}
=\frac{\varepsilon}{2}.\]
 Note that $f(x)=0 $
when $x \in [ p,q ]$ ($p$ and $q$ are in $Q_f$). For $j$ such that
$\tau_k+j-1$ is a time before the exit time of the interval
$[p,q]$, we have
\[U_{\tau_k+j}=X_{\tau_k}+W_{\tau_k+j}-W_{\tau_k}.\]
As $|W_{\tau_k+j}-W_{\tau_k}|<\frac{\varepsilon}{4},$ we have
$U_{\tau_k+j}\leq
p+\frac{\varepsilon}{2}+\frac{\varepsilon}{4}\leq p+\varepsilon,$
the precess will never exceed $p+\varepsilon$ before
$\sigma_{k+1}$. We conclude that $\sup_{k\geq n} U_k\leq
p+\varepsilon.$\\
To establish that the limit is to a stable point, we refer the
reader to \cite{Renlund1} to see a detailed proof.

\end{proof}

\begin{thm}[\cite{Renlund2}]\label{clt-renlund}
Let $(U_n)_{n\geq 0}$ satisfying Equation~\eqref{eq:algo_sto} such
that $\displaystyle\lim_{n\to +\infty} U_n = U^\star$. Let $\hat
\gamma_n:= n\gamma_n \hat f(U_{n-1})$ where $\hat f(x) =
\frac{-f(x)}{x-U^\star}$. Assume that $\hat \gamma_n$ converges
almost surely to some limit $\hat \gamma$. Then,

if $\hat\gamma > \frac12$ and if $\mathbb E[(n\gamma_n \Delta
M_n)^2|\mathcal F_{n-1}] \to \sigma^2 > 0$, then \[\sqrt n (U_n
-U^\star) \to \mathcal N\Big(0, \frac{\sigma^2}{2\hat\gamma
-1}\Big).\]

\end{thm}

\subsection{Proof of the main results}
\begin{proof}[Proof of Theorem \ref{thmXopp}] Consider the
urn model defined in Equation (\ref{recurrence}) with $Q_n=
\begin{pmatrix}
  0 & X_n \\
  X_n & 0 \\
\end{pmatrix}$.
We have the following recursions:
\begin{equation}\label{recurrence-opp2}W_{n+1}=W_n+X_{n+1}(m-\xi_{n+1})\quad \text{and} \quad T_{n+1}=T_n+mX_{n+1}.\end{equation}
\textbf{Proof of claim 1}
\begin{lem}
Let $Z_n=\frac{W_n}{T_n}$ be  the proportion of white balls in the
urn after $n$ draws. Then, $Z_n $ satisfies the stochastic
approximation algorithm defined by (\ref{eq:algo_sto}) with
$\gamma_n=\frac{1}{T_n}$, $f(x)=\mu_X m(1-2x)$ and $\Delta
M_{n+1}=X_{n+1}(m-\xi_{n+1}-mZ_n)-\mu m(1-Z_n)$.
\end{lem}
\begin{proof}
We need to check the conditions of definition \ref{def-algo}.\\

\begin{description}
\item [(i)]  Recall that $T_n=T_0+m\sum_{i=1}^nX_i$, with
$(X_i)_{i\geq 1}$ are iid random variables. It follows, by
Rajechman strong law of large numbers, that
\begin{equation}\label{asymp-T_n}T_n\stackrel{a.s}{=}\mu_X mn
+o(\sqrt{n}\ \ln(n)^\delta),\quad \delta
>\frac{1}{2},\end{equation} it follows that
$\frac{1}{(m\mu_X+1)n}\leq\frac{1}{T_n}\leq\frac{2}{m\mu_X n},$
then, $c_l=\frac{1}{m\mu_X+1}$ and $c_u=\frac{2}{m\mu_X n},$

\item[(ii)] $\mathbb{E}[\Delta
M_{n+1}^2|\mathcal{F}_n]\leq (6m^2+m)\mathbb{E}(X^2)+9m^2\mu^2=K_u,$\\

\item[(iii)] $|f(Z_n)|=m\mu_X|1-2Z_n|\leq 3m\mu_X=K_f$,\\

\item[(iv)]$\mathbb{E}(\gamma_{n+1}\Delta
M_{n+1}|\mathcal{F}_n)\leq \frac{1}{T_n}\mathbb{E}(\Delta
M_{n+1}|\mathcal{F}_n)=0=K_e$.
\end{description}

\end{proof}

\begin{prop}\label{prop1/2}
 The proportion of white balls in the urn
after $n$ draws, $Z_n$,  converges almost surely to $\frac{1}{2}$.
\end{prop}

\begin{proof}[Proof of Proposition \ref{prop1/2}]

Since the process $Z_n$ satisfies the stochastic approximation
algorithm defined by Equation (\ref{eq:algo_sto}), we apply
Theorem \ref{th:renlund}. As the function $f$ is continuous we
conclude that $Z_n$ converges almost surely to $\frac{1}{2}$: the
unique stable zero of the function $f$.
\end{proof}
We apply the previous results to the urn composition.  As we can
write $\frac{W_n}{n}=\frac{W_n}{T_n}\frac{T_n}{n}$, we deduce from
Proposition \ref{prop1/2} and Equation (\ref{asymp-T_n}) that
$\frac{W_n}{n}\stackrel{a.s}{=}\bigl(\frac{1}{2}+o(1)\bigr)\Big(\mu_Xm+o\Bigl(\frac{\ln(n)^\delta}{\sqrt{n}}\Bigr)\Bigr),$
then this corollary follows:
\begin{cor}
The number of white balls in the urn after $n$ draws, $W_n$,
satisfies for $n$ large enough
\begin{equation*}W_n\stackrel{a.s}{=}\frac{\mu_X m}{2}n+o(\sqrt{n}\ \ln(n)^\delta),\quad \delta >\frac{1}{2}.\end{equation*}
\end{cor}
\textbf{Proof of claim 2} We aim to apply Theorem
\ref{clt-renlund}. For this reason, we need to find this limits:
\begin{equation*}
    \lim_{n\rightarrow \infty}\mathbb{E}[\bigl(\frac{n}{T_n}\bigr)^2\Delta
    M_{n+1}^2|\mathcal{F}_n]\quad \text{and}\quad
    \lim_{n\rightarrow \infty}-\frac{n}{T_n}f'(Z_n).
\end{equation*}

We have
\begin{eqnarray*}
\mathbb{E}[\Delta
M_{n+1}^2|\mathcal{F}_n]&=&\mathbb{E}(X_{n+1}^2)\mathbb{E}[(m-\xi_{n+1}-mZ_n)^2|\mathcal{F}_n])+\mu^2\mathbb{E}[(m-2mZ_n)^2|\mathcal{F}_n]
\\ &&-2\mu_X^2\mathbb{E}[(m-\xi_{n+1}-mZ_n)(m-2mZ_n)|\mathcal{F}_n]\\
&=&(\sigma_X^2+\mu_X^2)\Big[m^2-4m^2Z_n+4m^2Z_n^2+mZ_n(1-Z_n)\frac{T_n-m}{T_n-1}\Big]-\mu_X^2[m^2+4m^2Z_n^2-4m^2Z_n].
\end{eqnarray*}
As $n$ tends to infinity, we have $Z_n
\stackrel{a.s}{\longrightarrow} \frac{1}{2}$ and
$\frac{T_n-m}{T_n-1}\stackrel{a.s}{\longrightarrow} 1$. Then,
\begin{equation*}\lim_{n\rightarrow \infty}\mathbb{E}[\Delta
M_{n+1}^2|\mathcal{F}_n]\stackrel{a.s}{=}(\sigma_X^2+\mu_X^2)\frac{m}{4}\quad
\text{and}\quad \lim_{n\rightarrow
\infty}-\frac{n}{T_n}f'(Z_n)\stackrel{a.s}{=}2.\end{equation*}
According to Theorem \ref{clt-renlund},
$\sqrt{n}(Z_n-\frac{1}{2})$ converges in distribution to
$\mathcal{N}(0,\frac{\sigma_X^2+\mu_X^2}{12\mu_X^2m})$. Finally,
by writing
$\Big(\frac{W_n-\frac{1}{2}T_n}{\sqrt{n}}\Big)=\sqrt{n}(Z_n-\frac{1}{2})\frac{T_n}{n}$,
we conclude using Slutsky theorem.\\

\textbf{Proof of claim 3} To prove this claim, we follow the proof
of Lemma 3 and Theorem 2 in  \cite{A.L.O}. Using the same methods,
we show in a first step that the variables $(X_n(m-\xi_n))_{n\geq
0}$ are $\alpha$-mixing variables with a strong mixing coefficient
$\alpha(n)=o\Big(\frac{\ln(n)^\delta}{\sqrt{n}}\Big)$, $\delta
>\frac{1}{2}$. To conclude, we adapt the Bernstein method.
Consider the same notation as in Theorem 2 in \cite{A.L.O}, and
define $S_n=\frac{1}{\sqrt{n}}\sum_{i=1}^n\tilde\xi_i$ where
$\tilde\xi_i=X_i(m-\xi_i)-\mu_X(m-\mathbb{E}(\xi_i))$. At first,
we need to estimate the variance of $W_n$.
\begin{prop}\label{var}
The variance of $W_n$ satisfies
\begin{equation}\label{variance}\mathbb{V}ar(W_n)=\frac{m(\sigma_X^2+\mu_X^2)+m^2\sigma_X^2}{12}\ n+o(\sqrt{n}\ \ln(n)^\delta),\quad \delta>\frac{1}{2}.\end{equation}
\end{prop}

\begin{proof}[Proof of Proposition \ref{var}]
Recall that the number of white balls in the urn satisfies
Equation (\ref{recurrence-opp2}), then
\begin{equation*}\mathbb{V}ar(W_{n+1})=\mathbb{V}ar(W_n)+\mathbb{V}ar(X_n(m-\xi_n))+2\ \mathbb{C}ov(W_{n-1},X_n(m-\xi_n)).\end{equation*}
We have
$\mathbb{V}ar(X_n(m-\xi_n))=(\sigma_X^2+\mu_X^2)\Big(\mathbb{V}ar(mZ_{n-1})
+\mathbb{E}\Big(mZ_{n-1}(1-Z_{n-1})\frac{T_{n-1}-m}{T_{n-1}-1}\Big)\Big)+\sigma_X^2\mathbb{E}(m-\xi_n)^2.$

Using Equation (\ref{asymp-T_n}) and the fact that
$Z_n\stackrel{a.s}{\rightarrow}\frac{1}{2}$, we obtain
\begin{eqnarray*}\mathbb{V}ar(W_{n+1})&=&\Big(1-\frac
2n+o\Big(\frac{\ln(n)^\delta}{n^{\frac32}}\Big)\Big)\mathbb{V}ar(W_{n})
+\frac{m(\sigma_X^2+\mu_X^2)+m^2\sigma_X^2}{4}+o\Big(\frac{\ln(n)^\delta}{\sqrt
n}\Big)\\
&=&a_n\mathbb{V}ar(W_{n})+b_n,\end{eqnarray*} where
$a_n=\Bigl(1-\frac
2n+o\Big(\frac{\ln(n)^\delta}{n^{\frac32}}\Big)\Bigr)$ and
$b_n=\frac{m(\sigma_X^2+\mu_X^2)+m^2\sigma_X^2}{4}+o\Big(\frac{\ln(n)^\delta}{\sqrt
n}\Big).$\\
Thus, \begin{equation*} \mathbb{V}ar(W_n)=\Big(\prod_{k=1}^n
a_k\Big)\Big(\mathbb{V}ar(W_0)+\sum_{k=0}^{n-1}\frac{b_k}{\prod_{j=0}^ka_j}\Big).
\end{equation*}
There exists a constant $a$ such that

$\prod_{k=1}^na_k=\displaystyle\frac{e^{a}}{n^2}\Big(1+o\Big(\frac{\ln(n)^\delta}{\sqrt
n}\Big)\Big)$, which leads to
\begin{equation*}
\mathbb{V}ar(W_n)=\frac{m(\sigma_X^2+\mu_X^2)+m^2\sigma_X^2}{12}n+o(\sqrt
n\ln(n)^\delta),\quad \delta>\frac{1}{2}.
\end{equation*}
\end{proof}
Recall that we follow the proof of Theorem 2 in \cite{A.L.O},
using Equation (\ref{variance}), we conclude that
\begin{equation}\frac{W_n-\mathbb{E}(W_n)}{\sqrt{n}}\stackrel{\mathcal{L}}{\longrightarrow}
 \mathcal{N}\Bigl(0,\frac{m(\sigma_X^2+\mu_X^2)+m^2\sigma_X^2}{12}\Bigr).
 \end{equation}
\end{proof}

 \begin{proof}[Proof of Theorem \ref{thmXself}]

Consider the urn model defined in (\ref{recurrence}) with $Q_n=
\begin{pmatrix}
  X_n & 0 \\
  0 & X_n \\
\end{pmatrix}$. The following recurrences hold:

\begin{equation}\label{W-self}
W_{n+1}=W_n+X_{n+1}\xi_{n+1}\quad \text{and}\quad
T_{n+1}=T_n+mX_{n+1}.
\end{equation}
As $T_n$ is a sum of iid random variables then $T_n$ satisfies the
following
\begin{equation}\label{totalself}T_n\stackrel{a.s}{=}\frac{\mu_Xm}{2}n+o(\sqrt{n}\ln(n)^\delta).\end{equation}
 The processes $\tilde
M_{n}=\prod_{k=1}^{n-1}\Big(\frac{T_k}{T_k+m\mu_X}\Big)W_n$ and
$\tilde N_n=\prod_{k=1}^{n-1}\Big(\frac{T_k}{T_k+m\mu_X}\Big)B_n$
are two $\mathcal{F}_n$ positive martingales. In view of
(\ref{totalself}), we have
$\prod_{k=1}^{n-1}\Big(\frac{T_k}{T_k+m\mu_X}\Big)\stackrel{a.s}{=}\displaystyle\frac{e^{\gamma}}{n}\Big(1+o\Big(\frac{\ln
(n)^\delta}{\sqrt n}\Big)\Big)$ for a positive constant $\gamma$.
Thus, there exists nonnegative random variables $\tilde W_\infty$
and $\tilde B_{\infty}$  such that $\tilde W_\infty+\tilde
B_\infty\stackrel{a.s}{=}m\mu_X$ and

\begin{equation*}\frac{W_n}{n}\stackrel{a.s}{\longrightarrow}\tilde W_\infty,\quad
\text{and}\quad \frac{B_n}{n}\stackrel{a.s}{\longrightarrow}\tilde
B_{\infty}.\end{equation*}

\textbf{Example: } In the original P\`olya urn model \cite{Polya},
when $m=1$ and $X=C$ (deterministic), the random variable $\tilde
W_\infty/C$ has a $Beta(\frac{B_0}{C},\frac{W_0}{C})$ distribution
 \cite{Athreya-Ney,S. Janson}. Whereas, M.R. Chen and M. Kuba
\cite{Chen-Kuba} considered the case when $X=C$ (non random) and
$m>1$. They gave moments of all orders of $W_n$ and proved that
$\tilde W_\infty$ cannot be an ordinary $Beta$ distribution.

\begin{rmq}
 Suppose that the
random variable $X$ has moments of all orders,  let
$\:m_k=E(X^k)$, for $ k\ge 1$. We have, almost surely, $W_n\le
T_n$ then, by Minskowski inequality, we obtain
$\mathbb{E}(W_n^{2k})\leq (mn)^{2k}\mathbb{E}(X^{2k})$. Using
Carleman's condition we conclude that, if
$\sum_{k\ge1}\mu_{2k}^{-\frac{1}{2k}}=\infty$, then the random
variable $\tilde W_\infty$ is determined by its moments.
Unfortunately, till now we still unable to give exact expressions
of moments of all orders of $W_n$. But, we can characterize the
distribution of $\tilde W_\infty$ in the case when the variable
$X$ is bounded.
\end{rmq}

\begin{lem}
\label{Abs_con} Assume that $X$ is a bounded random variable,
then, for fixed $W_0,B_0$ and $m$ the random variable $\tilde
W_\infty$ is absolutely continuous.
\end{lem}

The proof that $\tilde W_\infty$ is absolutely continuous is very
close to that of Theorem 4.2 in \cite{Chen-Wei}. We give the main
proposition to make  the proof clearer.

\begin{prop}\cite{Chen-Wei}
Let $\Omega_\ell$ be a sequence of increasing events such that
$\mathbb{P}(\cup_{\ell \ge 0}\Omega_\ell)=1$. If there exists
nonnegative Borel measurable function $\{f_\ell\}_{\ell\geq 1}$
such that $\mathbb{P}\Big(\Omega_\ell\cap \tilde
W_\infty^{-1}(B)\Big)=\int_Bf_\ell (x)dx$ for all Borel sets B,
then, $f=\displaystyle\lim_{l\rightarrow+\infty}f_\ell$ exists
almost everywhere and $f$ is the density of $\tilde W_\infty$.
\end{prop}

  Let
$(\Omega,\mathcal F,\mathcal{P})$ be a probability space. Suppose
that there exists a constant $A$ such that, we have almost surely,
$X\le A$. \begin{lem} Define the events
\begin{equation*}
\Omega_{\ell}:=\{W_\ell\ge m A \;\mbox{and}\;B_\ell\ge mA\},
\end{equation*}
then, $(\Omega_{\ell})_{\ell\geq 0}$ is a sequence of increasing
events, moreover we  have $\mathbb{P}(\cup_{\ell \ge
0}\Omega_\ell)=1$.
\end{lem}

Next, we just need to show that the restriction of $\tilde
W_\infty$ on $\Omega_{\ell,j}=\{\omega; W_\ell(\omega)=j\}$ has a
density for each $j$, with $Am\leq j\leq T_{\ell-1}.$ Let
$(p_c)_{c\in\text{supp}(X)}$  the distribution of $X$.


\begin{lem}
For a fixed $\ell>0$, there exists a positive constant $\kappa$,
such that, for every $c\in\text{supp(X)}$, $n\ge \ell+1$, $Am\le
j\le T_{\ell-1}$ and $k\le Am(n+1)$, we have

\begin{equation}
\label{Inequality_WEI} \sum_{i=0}^m
\mathbb{P}(W_{n+1}=j+k|W_n=j+k-ci)\le p_c(1-\frac
1n+\frac{\kappa}{n^2}).
\end{equation}

\end{lem}
\begin{proof}
According to  Lemma 4.1 \cite{Chen-Wei},
 for $Am \leq
j\leq T_{\ell -1}$, $n\geq \ell$ and $k\leq Am(n+1)$, the
following holds:
\begin{equation}\label{step2}\sum_{i=0}^m{j+c(k-i)\choose
i}{T_n-j-c(k-i)\choose
m-i}=\frac{T_n^m}{m!}+\frac{(1-m-2c)T_n^{m-1}}{2(m-1)!}+...,\end{equation}
which is a polynomial in $T_n$ of degree $m$ with coefficients
depending  on $W_0, B_0, m$ and $c$ only.\\

Let $u_{n,k}(c)=\sum_{i=0}^m \mathbb{P}(W_{n+1}=j+k|W_n=j+k-ic)$.
Applying Equation (\ref{step2}) to our model we have
\begin{eqnarray}
\label{Majoration1} u_{n,k}(c)&=&p_c\sum_{i=0}^m{j+k\choose
i}{T_n-j-k\choose m-i}{T_n\choose m}^{-1} \nonumber
\\
&=&p_c{T_n\choose m}^{-1}\Big(\frac{T_n^{m}}{m!}+
\frac{(1-m-2c)}{(m-1)!}T_n^{m-1}+\ldots\Big)\Big(\frac{T_n^m}{m!}+\frac{(1-m)}{2(m-1)!}T_n^{m-1}+\ldots\Big)^{-1}\nonumber
\\
&\stackrel{a.s}{=}&
p_c\Big(1-\frac{1}{n}+O\Big(\frac{1}{n^2}\Big)\Big).
\end{eqnarray}
\end{proof}

 Later, we will limit the proof by mentioning the main differences with
 Lemma 4.1 \cite{Chen-Wei}. For a fixed $\ell$ and $n\ge \ell+1$, we denote by $v_{n,j}=\displaystyle\max_{0\leq k\leq
Amn}\mathbb{P}\bigl(W_{\ell+n}=j+k|W_\ell=j\bigr)$.  We have the
following inequality:
\begin{eqnarray*}
v_{n+1,j}&\le & \max_{0\le k\le
Am(n+1)}\Big\{\sum_{i=0}^m\sum_{c\in\text{supp}{(X)}}\mathbb{P}(W_{\ell+
n+1}=j+k|W_{\ell+n}=j+k-ci)\Big\}\nonumber
\\
&\le& \max_{0\le k\le Am(n+1)}\Big\{\sum_{i=0}^m\sum_{c\in\text{supp}{(X)}}\mathbb{P}(W_{\ell+n+1}=j+k|W_{\ell+n}=j+k-ci)\nonumber\\
&&\times \mathbb{P}(W_{\ell+n}=j+k-ci|W_\ell=j)\Big\}\nonumber\\
&\le&\max_{0\le k\le
Am(n+1)}\sum_{i=0}^m\sum_{c\in\text{supp}{(X)}}\mathbb{P}(W_{\ell+n+1}=j+k|W_{\ell+n}=j+k-ci)\\
&&\times \max_{0\leq \tilde k\leq
Amn}\mathbb{P}\bigl(W_{\ell+n}=j+\tilde
k|W_\ell =j\bigr)\\
&\le&\sum_{c\in\text{supp}{(X)}}p_c\Big(1-\frac{1}{n+l}+\frac{\kappa}{(n+l)^2}\Big)v_{n,j}\\&&=\Big(1-\frac{1}{n+l}+\frac{\kappa}{(n+l)^2}\Big)v_{n,j}.
\end{eqnarray*}
This implies that there exists some positive constant $C(\ell)$,
depending on $\ell$ only, such that, for a fixed $\ell$ and for
all $n\ge \ell+1$, we get
\begin{equation}
\max_{0\leq k\leq
m(n-l)}\mathbb{P}\bigl(W_n=j+k|W_l=j\bigr)\le\prod_{i=\ell}^n\Big(1-\frac1i+\frac{\kappa}{i^2}\Big)\le
\frac{C(\ell)}{n}.
\end{equation}
The rest of the proof follows.

\end{proof}

 \begin{proof}[Proof of Theorem \ref{thmXYopp}]
Consider the urn  model evolving by the matrix $Q_n=
\begin{pmatrix}
 0 & X_n \\
  Y_n & 0 \\
\end{pmatrix}$. According to Equation (\ref{recurrence}), we have
the following recursions:
\begin{equation}\label{opposite-rec}W_{n+1}=W_n+X_{n+1}(m-\xi_{n+1})\quad \text{and}\quad T_{n+1}=T_n+mX_{n+1}+\xi_{n+1}(Y_{n+1}-X_{n+1}).\end{equation}

\begin{lem}
 The proportion of white balls after $n$ draws, $Z_n$, satisfies the stochastic algorithm
defined by (\ref{eq:algo_sto}), where
$f(x)=m(\mu_Y-\mu_X)x^2-2\mu_Xmx+\mu_Xm$, $\gamma_n=\frac{1}{T_n}$
and $\Delta M_{n+1}=D_{n+1}-\mathbb{E}[D_{n+1}|\mathcal{F}_n]$,
 with $D_{n+1}=\xi_{n+1}(Z_n(X_{n+1}-Y_{n+1})-X_{n+1})+mX_{n+1}$.
 \end{lem}
 \begin{proof}
 We check the conditions of Definition \ref{def-algo}, indeed,
 \begin{description}
 \item[(i)] recall that
 $T_n=T_0+m\sum_{i=1}^nX_i+\sum_{i=1}^n\xi_i(Y_i-X_i)$, then $\frac{T_n}{n}\leq\frac{T_0}{n}+\frac{m}{n}\sum_{i=1}^nX_i+\frac{m}{n}\sum_{i=1}^n|Y_i-X_i|.$
 By the strong law of large numbers we have $\frac{T_n}{n}\leq
 m(\mu_X+\mu_{|Y-X|})+1$. On the other hand, we have  $T_n\geq \displaystyle\min_{1\leq i\leq n}(X_i, Y_i) m n,$
   thus, the following bound holds
   \begin{equation*}\frac{1}{(m(\mu_X+\mu_{|Y-X|})+1)n}\leq \frac{1}{T_n}\leq \frac{1}{m \displaystyle\min_{1\leq i\leq n}(X_i,Y_i) n},\end{equation*}
   then $c_l=\frac{1}{(m(\mu_X+\mu_{|Y-X|})+1)n}$ and $c_u=\frac{1}{m \displaystyle\min_{1\leq i\leq n}(X_i,Y_i)},$\\
 \item[(ii)] $\mathbb{E}[\Delta M_{n+1}^2|\mathcal{F}_n]
  \leq
  (\mu_{(X-Y)^2}+3\mu_X)(m+m^2)+5m^2\mu_{X^2}+2m^2\mu_X\mu_Y+m^2(|\mu_X-\mu_Y|+3\mu_X)=K_u,$
 \item[(iii)]$|f(Z_n)|\leq
m(|\mu_Y-\mu_X|+3\mu_X)=K_f,$
 \item[(iv)] $\mathbb{E}[\frac{1}{T_{n+1}}\Delta
M_{n+1}|\mathcal{F}_n]\leq \frac{1}{T_n}\mathbb{E}[\Delta
M_{n+1}|\mathcal{F}_n]=0$
 \end{description}
 \end{proof}

 \begin{prop}
The proportion of white balls in the urn after $n$ draws, $Z_n$,
satisfies as $n$ tends to infinity
\begin{equation}\label{conv-proportion}Z_n\stackrel{a.s}{\longrightarrow}z:=\frac{\sqrt{\mu_X}}{\sqrt{\mu_X}+\sqrt{\mu_Y}}.\end{equation}
\end{prop}

\begin{proof}

The proportion of white balls in the urn satisfies the stochastic
approximation algorithm defined in (\ref{eq:algo_sto}). As the
function $f$ is continuous, by Theorem \ref{th:renlund}, the
process $Z_n$ converges almost surely to
 $z=\frac{\sqrt{\mu_X}}{\sqrt{\mu_X}+\sqrt{\mu_Y}}$,
 the unique zero of $f$ with negative derivative.

\end{proof}
Next, we  give an estimate of $T_n$, the total number of balls in
the urn after $n$ draws, in order to describe the asymptotic of
the urn composition. By Equation (\ref{opposite-rec}), we have

\begin{equation*}\frac{T_n}{n}=\frac{T_0}{n}+\frac{m}{n}\sum_{i=1}^nX_i
+\frac{m(\mu_Y-\mu_X)}{n}\sum_{i=1}^nZ_{i-1}+\frac{1}{n}\sum_{i=1}^n\Big[\xi_i(Y_i-X_i)-\mathbb{E}[\xi_i(Y_i-X_i)|\mathcal{F}_{i-1}]\Big].\end{equation*}
Since $(X_i)_{i\geq 1}$ are iid random variables, then by the
strong law of large numbers we have
$\frac{m}{n}\sum_{i=1}^nX_i\stackrel{a.s}{\rightarrow} m\mu_X$.
Via Ces\'aro lemma, we conclude that
$\frac{1}{n}\sum_{i=1}^nZ_{i-1}$ converges almost surely, as $n$
tends to infinity, to $z$. Finally, we prove that last term in the
right side tends to zero, as $n$ tends to infinity. In fact, let
$G_n=\sum_{i=1}^n\Big[\xi_i(Y_i-X_i)-\mathbb{E}[\xi_i(Y_i-X_i)|\mathcal{F}_{i-1}]\Big]$,
then $(G_n,\mathcal{F}_n)$ is a martingale difference sequence
such that
\begin{equation*}\frac{<G>_n}{n}=\frac{1}{n}\sum_{i=1}^n\mathbb{E}[\nabla G_i^2|\mathcal{F}_{i-1}],\end{equation*}
where $\nabla
G_n=G_n-G_{n-1}=\xi_n(Y_n-X_{n})-\mathbb{E}[\xi_n(Y_n-X_{n})|\mathcal{F}_{n-1}]$ and $<G>_n$ denotes the quadratic variation of the martingale.\\
By a simple computation, we have the almost sure convergence of
$\mathbb{E}[\nabla G_i^2|\mathcal{F}_{i-1}]$ to $(mz
(1-z)+m^2z^2)(\sigma_Y^2+\sigma_X^2)$. Therefore, Ces\'aro lemma
ensures that, $\frac{<G>_n}{n}$ converges to $(mz
(1-z)+m^2z^2)(\sigma_Y^2+\sigma_X^2)$ and
$\frac{G_n}{n}\stackrel{a.s}{\longrightarrow} 0$. Thus, for $n$
large enough we have
\begin{equation}\label{T_n-convergence}\frac{T_n}{n}\stackrel{a.s}{\longrightarrow}
m\sqrt{\mu_X}\sqrt{\mu_Y}.\end{equation} In view of Equation
(\ref{T_n-convergence}), we describe the asymptotic behavior of
the  urn composition after $n$ draws. One can write
$\frac{W_n}{n}\frac{W_n}{T_n}\frac{T_n}{n}$ and
$\frac{B_n}{n}\stackrel{a.s}{=}\frac{B_n}{T_n}\frac{T_n}{n}$,
using Equations (\ref{conv-proportion}, \ref{T_n-convergence}) and
Slutsky theorem, we have, as $n$ tends to infinity,
$\frac{W_n}{n}\stackrel{a.s}{\longrightarrow}m\sqrt{\mu_X}\sqrt{\mu_Y}
z$ and  $\frac{B_n}{n}\stackrel{a.s}{\longrightarrow}
m\sqrt{\mu_X}\sqrt{\mu_Y}(1-z)$.\\
\textbf{Proof of claim 2}\\
 Later, we aim to apply Theorem \ref{clt-renlund}. In our model, we have $\gamma_n=\frac{1}{T_n}$, then we need to
control the following asymptotic behaviors
\begin{equation*}
    \lim_{n\rightarrow +\infty}\mathbb{E}[\Big(\frac{n}{T_n}\Big)^2\Delta
    M_{n+1}^2|\mathcal{F}_n]\quad \text{and}\quad
    \lim_{n\rightarrow +\infty}-\frac{n}{T_n}f'(Z_n).
\end{equation*}
In fact, recall that $\frac{n}{T_n}$ converges almost surely to
$\frac{1}{m\sqrt{\mu_X}\sqrt{\mu_Y}}$ and $\mathbb{E}[\Delta
M_{n+1}^2|\mathcal{F}_n]=\mathbb{E}[D_{n+1}^2|\mathcal{F}_n]+\mathbb{E}[D_{n+1}|\mathcal{F}_n]^2$.
Since $\mathbb{E}[D_{n+1}|\mathcal{F}_n]^2$ converges almost
surely to $f(z)^2=0$, we have,
\begin{eqnarray*}\mathbb{E}[D_{n+1}^2|\mathcal{F}_n]&=&\mathbb{E}\Big[Z_n^2(X_{n+1}-Y_{n+1})^2-2Z_nX_{n+1}+X_{n+1}|\mathcal{F}_n\Big]
\mathbb{E}[\xi_{n+1}^2|\mathcal{F}_n]+m^2\mathbb{E}(X^2)\\&&+2m^2\Bigl(Z_n^2(\mathbb{E}(X^2)-\mu_X\mu_Y)-Z_n\mathbb{E}(X^2)\Bigr).\end{eqnarray*}
Using the fact that
$\mathbb{E}[\xi_{n+1}^2|\mathcal{F}_n]=mZ_n(1-Z_n)\frac{T_n-m}{T_n-1}+m^2Z_n^2$
and that $Z_n$ converges almost surely to $z$, we conclude that
$\mathbb{E}[D_{n+1}^2|\mathcal{F}_n]$ converges almost surely to
$G(z)>0.$  Applying Theorem \ref{clt-renlund}, we obtain the
following
\begin{equation}\sqrt{n}(Z_n-z)\stackrel{\mathcal{L}}{\longrightarrow}\mathcal{N}\Big(0,\frac{G(z)}{3m^2\mu_X\mu_Y}\Big). \end{equation}
But, we can write
$\frac{W_n-zT_n}{\sqrt{n}}=\sqrt{n}\bigl(\frac{W_n}{T_n}-z\bigr)\frac{T_n}{n}$.
Thus, it is enough to use Slutsky theorem to conclude the proof.

\end{proof}

\begin{proof}[Proof of Theorem \ref{thmXYself}] Consider
the urn model defined in (\ref{recurrence}) with $Q_n=
\begin{pmatrix}
  X_n &  0\\
  0& Y_n \\
\end{pmatrix}$.  The process of the
urn satisfies the following recursions:
\begin{equation}\label{recurence-self1}
W_{n+1}=W_n+X_{n+1}\xi_{n+1}\quad \text{and} \quad
T_{n+1}=T_n+mY_{n+1}+\xi_{n+1}(X_{n+1}-Y_{n+1}).\end{equation}
\begin{lem}\label{algo}
If $\mu_X\neq \mu_Y$, the proportion of white balls in the urn
after $n$ draws satisfies the stochastic algorithm defined by
(\ref{eq:algo_sto}) where $\gamma_n=\frac{1}{T_n}$,
$f(x)=m(\mu_Y-\mu_X)x(x-1)$ and $\Delta
M_{n+1}=D_{n+1}-\mathbb{E}[D_{n+1}|\mathcal{F}_n]$ with
$D_{n+1}=\xi_{n+1}(Z_n(Y_{n+1}-X_{n+1})+X_{n+1})-mZ_nY_{n+1}$.\\
\end{lem}
\begin{proof}
We check that, if $\mu_X\neq\mu_Y$, the conditions of definition
\ref{def-algo} hold. Indeed, \begin{description} \item[(i)] as
$T_n=T_0+m\sum_{i=1}^nY_i+\sum_{i=1}^n\xi_i(X_i-Y_i)$,
 then via the strong law of large numbers we have $|\frac{T_n}{n}|\leq
 m\mu_Y+m\mu_{|X-Y|}+1$.
On the other hand, we have $T_n\geq \min_{1\leq i\leq n}(X_i,Y_i)
m n$, thus,\begin{equation*}\frac{1}{(m\mu_Y+m\mu_{|X-Y|})n}\leq
\frac{1}{T_n}\leq \frac{1}{\displaystyle\min_{1\leq i\leq
n}(X_i,Y_i) m n},\end{equation*}

 \item[(ii)]$\mathbb{E}[\Delta
M_{n+1}^2|\mathcal{F}_n]\leq
(2m+m^2)(4\mu_{X^2}+\mu_{Y^2})+3m^2\mu_{Y^2}+2m^2\mu_X+2m^2\mu_X\mu_Y+4m^2(\mu_X-\mu_Y)^2=K_u,$\\
\item
[(iii)]$|f(Z_n)|=|m(\mu_Y-\mu_X)Z_n(Z_n-1)|\leq 2m |\mu_Y-\mu_X|=K_f,$\\
\item[(iv)] $\mathbb{E}[\gamma_{n+1}\Delta
M_{n+1}|\mathcal{F}_n]\leq \frac{1}{T_n}\mathbb{E}[\Delta
M_{n+1}|\mathcal{F}_n]=0=K_e.$
\end{description}
\end{proof}

\begin{prop}\label{prop-self}
The proportion of white balls in the urn after $n$ draws, $Z_n$,
satisfies almost surely\\

$\displaystyle\lim_{n\rightarrow \infty}Z_n=\left\{%
\begin{array}{ll}
     1, & \hbox{if $\mu_X>\mu_Y$;} \\
     0, & \hbox{if $\mu_X<\mu_Y$;} \\
     \tilde Z_\infty, & \hbox{if $\mu_X=\mu_Y$,} \\
\end{array}%
\right $\\
 where $\tilde Z_\infty$ is a positive random
variable.
\end{prop}

\begin{proof}[Proof of Proposition \ref{prop-self}]

Recall that, if $\mu_X\neq \mu_Y$, $Z_n$ satisfies the stochastic
algorithm defined in Lemma \ref{algo}. As the function $f$ is
continuous, by Theorem \ref{clt-renlund} we conclude that $Z_n$
converges almost surely to the stable zero of the function $h$
with a negative derivative,
which is $1$ if $\mu_X>\mu_Y$ and $0$ if $\mu_X<\mu_Y.$\\
 In the case  when $\mu_X=\mu_Y$, we have
$Z_{n+1}=Z_n+\frac{P_{n+1}}{T_{n+1}}$, where
$P_{n+1}=X_{n+1}\xi_{n+1}-Z_n\bigl(mY_{n+1}+\xi_{n+1}(X_{n+1}-Y_{n+1})\bigr)$.
Since $\mathbb{E}[P_{n+1}|\mathcal{F}_n]=0$, then $Z_n$ is a
positive
 martingale which converges almost surely to a positive random variable $\tilde
Z_\infty$.\\
As a consequence, we have
\begin{cor} The total number of balls in the urn, $T_n$,
satisfies as $n$ tends to infinity

 if $\mu_X\geq \mu_Y$
\begin{equation*}
\frac{T_n}{n}\stackrel{a.s}{\longrightarrow}m\mu_X.
\end{equation*}

\end{cor}
\begin{proof}
In fact, let
$M_n=\sum_{i=1}^n\xi_i(X_i-Y_i)-\mathbb{E}[\xi_i(X_i-Y_i)|\mathcal{F}_{i-1}],$
we have
\begin{eqnarray*}\frac{T_n}{n}&=&\frac{T_0}{n}+\frac{m}{n}\sum_{i=1}^nY_i+\frac{1}{n}\sum_{i=1}^n\xi_i(X_i-Y_i)\\
&=&\frac{T_0}{n}+\frac{m}{n}\sum_{i=1}^nY_i+
\frac{m(\mu_X-\mu_Y)}{n}\sum_{i=1}^nZ_{i-1}+\frac{M_n}{n}.
\end{eqnarray*}
As it was proved in the previous theorem, we show that, as $n$
tends to infinity, we have
$\frac{M_n}{n}\stackrel{a.s}{\longrightarrow} 0$. Recall that, if
$\mu_X>\mu_X$ , $Z_n$ converges almost surely to $1$. Then, using
Ces\'aro lemma, we obtain the limits requested. If $\mu_X=\mu_Y$,
we have $\frac{1}{n}\sum_{i=1}^nY_i$ converges to $\mu_Y$.
\end{proof}

Using the results above, the convergence of the normalized number
of white balls follows immediately. Indeed, if $\mu_X>\mu_Y$, we
have, as $n$ tends to infinity,
\[\frac{W_n}{n}=\frac{W_n}{T_n}\frac{T_n}{n}\stackrel{a.s}{\longrightarrow}m\mu_X,\]
Let $\tilde
G_n=\Bigl(\prod_{i=1}^{n-1}(1+\frac{m\mu_Y}{T_i})\Bigr)^{-1}B_n,$
then $(\tilde G_n,\mathcal{F}_n)$ is a positive martingale. There
exists a positive number $A$ such that
$\prod_{i=1}^{n-1}(1+\frac{m\mu_Y}{T_i})\simeq  A
   n^{\rho}$. Then, as $n$ tends to infinity we have
   \begin{equation*} \frac{B_n}{n^\rho}\stackrel{a.s}{\rightarrow} B_{\infty},\end{equation*}
where $B_\infty$ is a positive random variable.\\
If $\mu_X=\mu_Y$, the sequences
$\Bigl(\prod_{i=1}^{n-1}(1+\frac{m\mu_X}{T_i})\Bigr)^{-1}W_n$ and
$\Bigl(\prod_{i=1}^{n-1}(1+\frac{m\mu_Y}{T_i})\Bigr)^{-1}B_n$ are
$\mathcal{F}_n$ martingales such that
$\Bigl(\prod_{i=1}^{n-1}(1+\frac{m\mu_X}{T_i})\Bigr)^{-1}\simeq B
   n,$ where $B>0$, then, as $n$ tends to infinity,
    we have
    \begin{equation*}\frac{W_n}{n}\stackrel{a.s}{\rightarrow}
   W_{\infty} \quad \text{and}\quad \frac{B_n}{n}\stackrel{a.s}{\rightarrow}  \tilde B_{\infty},\end{equation*}
where $W_{\infty}$ and $\tilde B_{\infty}$ are positive random
variables satisfying $ \tilde B_{\infty}=m\mu_X-W_{\infty}.$

\end{proof}
\begin{rmq}
The case when $\mu_X<\mu_Y$ is obtained by interchanging the
colors. In fact we have the following results:
\begin{equation*}T_n\stackrel{a.s}{=}m\mu_Y n+o(n),\quad W_n=\tilde W_\infty n^\sigma+o(n)\quad \text{and} \quad B_n=m\mu_Yn+o(n),\end{equation*}
where $\tilde W_\infty$ is a positive random variable and
$\sigma=\frac{\mu_X}{\mu_Y}.$
\end{rmq}


\begin{thebibliography}{99}
\bibitem{C.N.O}{\sc C. Mailler, N. Lasmer and O. Selmi}.(2017).
 Multiple drawing multi-color urns by stochastic approximation.(to appear).
 \bibitem{Polya}{\sc F. Eggenberger and G. P\'olya}.(1923). \"{U}ber die statistik verkeletter vorge. {\it
Zeitschift f\"{u}r Angewandte Mathematik und Mechanic}, 1:279-289.
 \bibitem{Pages}{\sc G. Pag\'es and S. Laruelle.} (2015). Randomized urns models revisited using stochastic
approximation. {\it   Annals of Applied Probability,} (23)4:
1409-1436.
 \bibitem{Renlund1}{\sc H. Renlund.} (2010).
Generalized Polya urns via stochastic approximation. arxiv:
1002.3716v1.
\bibitem{Renlund2}{\sc H. Renlund.}(2011). Limit theorem for
Stochastic approxiamtion algorithm. arxiv :1102.4741v1.
\bibitem{Mahmoud}{\sc H. Mahmoud.} (2004). Random spouts as
 internet model and P\'olya processes. {\it Actainformatica,} 41:
 1-18.
\bibitem{Athreya-Ney}{\sc  K.B. Athreya and P.E. Ney.}(1972). Branching Processes. {\it Springer-Verlag,
 Berlin.}
 \bibitem{Wei}{\sc L.J. Wei.}(1978). An application of an urn model to the design of
  sequential controlled clinical trials. {\it Journal of American
 Statistics Association,}(73), 363:559-563.
\bibitem{Kuba-Mahmoud-Panholzer} {\sc M. Kuba, H. Mahmoud and A. Panholzer.}(2013).
Analysis of a generalized Friedman's urn with multiple drawings,
{\it Discrete Applied Mathematics,} 161,  2968-2984 .
\bibitem{Chen-Kuba} {\sc M. R Chen and M. Kuba.} (2013). On generalized Polya
urn models, {\it Theory of probability and its application}, 40,
1169-1186.
\bibitem{Chen-Wei}{\sc M. R Chen and C. Z Wei.}(2005). A new urn model. {\it  Applied  Probability,}
(42)4, 964-976.
\bibitem{kuba-Zulzbach}{\sc M. Kuba and H. Sulzbach.} (2016). On martingale tail sums in affine two-color urn models with multiple
drawings. arXiv:1509.09053.
\bibitem{Kuba-mahmoud}{\sc M. Kuba and H.  Mahmoud.} (2016). Two-colour balanced affine urn
models with multiple drawings I: Central limit theorem. arXiv
:1503.09069.
 \bibitem{Kuba-Mahmoud2} {\sc M. Kuba and H. Mahmoud.}
(2016). Two-colour balanced affine urn models with multiple
drawings II: large-index and triangular urns.  arXiv: 1509.09053.
\bibitem{bio}{\sc N.L. Johnson, S. Kotz.}(1977). Urn models and
 their application. {\it John Wiley & Son}.




\bibitem{R.Aguech} {\sc R. Aguech.}(2009).  Limit Theorems for Random Triangular Urns Schemes. {\it Journal of Applied Probability}, 46(3), 827-843.



 \bibitem{A.L.O}{\sc R. Aguech, N. Lasmer, O. Selmi.}(2017). A
 generalized urn model with multiple drawing and random addition.
 (to appear.)
 \bibitem{Goldman}{\sc R.N. Goldman.}(1985). Polya's urn model and computer aided geometry design.{\it SIAM Journal on Algebraic Discrete
 Methods,} 6(1), 1-28.

 \bibitem{S. Janson}{\sc S. Janson}.(2006).Limit theorems of
 triangular urn schemes.{\it Probability Theory and Related
 Fields,} 134(3), 417-452.




\end{thebibliography}
\end{document}